\documentclass{amsart}

\newtheorem{theorem}{Theorem}[section]
\newtheorem{lemma}[theorem]{Lemma}

\theoremstyle{definition}

\theoremstyle{remark}
\newtheorem{remark}[theorem]{Remark}

\numberwithin{equation}{section}

\begin{document}

\title{Jordan $\{g, h\}$-derivations on algebra of matrices}


\author{Arindam Ghosh}
\address{Department of Mathematics, Indian Institute of Technology Patna, Patna-801106}
\curraddr{}
\email{E-mail: arindam.pma14@iitp.ac.in}
\thanks{}

\author{Om Prakash$^{\star}$}
\address{Department of Mathematics, Indian Institute of Technology Patna, Patna-801106}
\curraddr{}
\email{om@iitp.ac.in}
\thanks{* Corresponding author}



\date{}

\dedicatory{}

\begin{abstract}
In this article, we show that every Jordan $\{g, h\}$-derivation over $\mathcal{T}_n(C)$ is a $\{g, h\}$-derivation under an assumption, where $C$ is a commutative ring with unity $1\neq 0$. We give an example of a Jordan $\{g, h\}$-derivation over $\mathcal{T}_n(C)$ which is not a $\{g, h\}$-derivation. Also, we study Jordan $\{g, h\}$-derivation over $\mathcal{M}_n(C)$.
\end{abstract}

\maketitle

\section{Introduction}
 Jordan derivation over prime rings was initiated by Herstein in 1957 \cite{her} and he proved that every Jordan derivation over a prime ring of characteristic not $2$ is a derivation. In 1975, Cusack established the same result for semiprime rings \cite{cus}. Let $A, B$ be two algebras over a commutative ring $C$. A linear map $d:A \rightarrow B$ is said to be a derivation if $d(xy)=d(x)y+xd(y)$, for all $x,y \in A$ and $d$ is said to be a Jordan derivation if $d(x^2)=d(x)x+xd(x)$ for all $x \in A$. When $C$ is 2-torsion free, an equivalent definition of Jordan derivation is that a linear map $d:A\rightarrow B$ is said to be a Jordan derivation if $d(x\circ y)=d(x)\circ y+x\circ d(y)$ for all $x,y \in A$. Throughout the paper, we assume $C$ is $2$-torsion free with unity $1\neq 0$. A ring $R$ is $2$-torsion free if $2a=0$ for some $a\in R$ implies $a=0$. Jordan derivation over rings and algebras are studied by many authors \cite{ben,bre,bresa,gho,sin,zha,zhan}, where they found every Jordan derivation over undertaken rings or algebras is a derivation. After that some generalizations of (Jordan) derivations have been introduced like (Jordan) left derivation, (Jordan) generalized derivation, (Jordan) $P$-derivation. Many results had been proved on those derivations over some rings and algebras \cite{bres,ghos,li,liwan,ma}.

Recently, in 2016 Bre\v{s}ar \cite{Bresar} introduced $\{g, h\}$ derivation and Jordan $\{g, h\}$ derivation. They have considered algebra over a field $F$ with char($F$)$\neq2$. We take the same definition with assuming the algebra over $C$. Let $A$ be a unital algebra over $C$ and $f, g, h : A \rightarrow A$ are linear maps. Then $f$ is said to be a $\{g, h\}$-derivation if
\begin{equation}
\label{gh1.1}
f (xy) = g(x)y + xh(y) = h(x)y + xg(y) ~\text{for all}~ x, y \in A.
\end{equation}
If $f = g = h$ in \eqref{gh1.1}, Then $f$ is a usual derivation. Now, $f$ is said to be a Jordan $\{g, h\}$-derivation if
\begin{equation}
\label{gh1.3}
f (x \circ y) = g(x) \circ y + x \circ h(y)~ \text{for all}~ x, y \in A.
\end{equation}
If $f = g = h$ in \eqref{gh1.3}, then $f$ is a usual Jordan derivation. Every $\{g, h\}$-derivation is Jordan $\{g, h\}$-derivation, but the converse is not true which has been shown by Bre\v{s}ar in example $2.1$\cite{Bresar}.  It is easy to see that, if every Jordan $\{g, h\}$-derivation over $A$ is a $\{g, h\}$-derivation, then every Jordan derivation on $A$ is a derivation. In 2016, Bre\v{s}ar established that every Jordan $\{g, h\}$-derivation of a semiprime algebra $A$ over a field $\mathbb{F}$ with char($\mathbb{F}$)$\neq 2$, is a $\{g, h\}$-
derivation \cite{Bresar}. \\
In this paper, we prove that every Jordan $\{g, h\}$-derivation over upper triangular matrix algebras $\mathcal{T}_n(C)$ is a $\{g, h\}$-derivation under some assumption. Next, we establish that $\{g, h\}$-derivation is the only Jordan $\{g, h\}$-derivation on full matrix algebras $\mathcal{M}_n(C)$. It is proved without taking any additional assumption as taken for $\mathcal{T}_n(C)$.

\section{Jordan \{g,h\} derivation on $\mathcal{T}_n(C)$}

Now, we characterize Jordan  $\{g, h\}$-derivation over upper triangular matrix algebras. $\mathcal{T}_n(C)$ represents the algebra of $n\times n$ upper triangular matrices over $C$ and $e_{ij}$ is the matrix unit whose $(i,j)$-th entry is $1$, $0$ elsewhere.

\begin{theorem}
\label{thm1}
Let $f$ be a Jordan $\{g, h\}$-derivation on $\mathcal{T}_n(C)$, $n\geq 2$, with  $f(e_{ii})=g(e_{ii})e_{ii}+e_{ii}h(e_{ii})$, for all $i=1,2,\dots,n$. Then $f$ is a $\{g, h\}$-derivation.
\end{theorem}

Before proving the theorem, we prove several lemmas. Let $f:A\rightarrow A$ be a Jordan $\{g, h\}$-derivation.

\begin{lemma}
\label{pro1}
Let $a\in A$ such that $f(a^2)=g(a)a+ah(a)$. Then $f(a^2)=h(a)a+ag(a)$.
\end{lemma}

\begin{proof}
From \eqref{gh1.3}, $f (a \circ a) = h(a) \circ a + a \circ g(a)$. Therefore, $f(a^2)=f (a \circ a)-f(a^2)=h(a)a+ag(a)$.
\end{proof}

\begin{lemma}
\label{pro2}
Let $a, b\in A$ such that $f(ab)=g(a)b+ah(b)=h(a)b+ag(b)$. Then $f(ba)=g(b)a+bh(a)=h(b)a+bg(a)$.
\end{lemma}

\begin{proof}
From \eqref{gh1.3}, $f (x \circ y) = g(x) \circ y + x \circ h(y)$, for all $x, y\in A$. Therefore, $f(ba)=f (b \circ a)-f(ab)=g(b) \circ a + b \circ h(a)-h(a)b-ag(b)=g(b)a+bh(a)$. Similarly, $f(ba)=f (a \circ b)-f(ab)=g(a) \circ b + a \circ h(b) - g(a)b+ah(b) =h(b)a+bg(a)$.
\end{proof}

Now, suppose $f:\mathcal{T}_n(C)\rightarrow \mathcal{T}_n(C)$ is a Jordan $\{g, h\}$-derivation.

First we prove
\begin{equation}
\label{gh2}
f (e_{ij}e_{kl}) = g(e_{ij})e_{kl} + e_{ij}h(e_{kl}) = h(e_{ij})e_{kl} + e_{ij}g(e_{kl}),
\end{equation}
which is equivalent to
\begin{equation}
\label{gh3}
f (e_{kl}e_{ij}) = g(e_{kl})e_{ij} + e_{kl}h(e_{ij}) = h(e_{kl})e_{ij} + e_{kl}g(e_{ij})~\text{(by Lemma \ref{pro2})}.
\end{equation}

Now, let
\begin{equation}
\label{gh4}
g(e_{ij})=\sum\limits_{1 \leq m \leq p \leq n}g_{mp}^{(ij)}e_{mp}~, ~\text{where} ~g_{mp}^{(ij)} \in C
\end{equation}

and \begin{equation}
\label{gh5}
h(e_{ij})=\sum\limits_{1 \leq m \leq p \leq n}h_{mp}^{(ij)}e_{mp}~, ~\text{where} ~h_{mp}^{(ij)} \in C.
\end{equation}

\begin{lemma}
\label{lem3}
$f (e_{ii}e_{jj}) =g(e_{ii}) e_{jj} + e_{ii} h(e_{jj})=h(e_{ii}) e_{jj} + e_{ii} g(e_{jj})$, for $i\neq j$.
\end{lemma}

\begin{proof}
Let $i\neq j$. Without loss of generality, let $i<j$. Since $f$ is a Jordan $\{g, h\}$-derivation on $\mathcal{T}_n(C)$,
\begin{equation}
\label{gh6}
\begin{aligned}
0&=f (e_{ii}\circ e_{jj}) = g(e_{ii})\circ e_{jj} + e_{ii}\circ h(e_{jj})\\
&=g_{1j}^{(ii)}e_{1j}+g_{2j}^{(ii)}e_{2j}+\dots+g_{jj}^{(ii)}e_{jj}+g_{jj}^{(ii)}e_{jj}+g_{j,j+1}^{(ii)}e_{j,j+1}+\dots+g_{jn}^{(ii)}e_{jn}\\
&+h_{ii}^{(jj)}e_{ii}+h_{i,i+1}^{(jj)}e_{i,i+1}+\dots+h_{in}^{(jj)}e_{in}+h_{1i}^{(jj)}e_{1i}+h_{2i}^{(jj)}e_{2i}+\dots+h_{ii}^{(jj)}e_{ii}.
\end{aligned}
\end{equation}

\begin{equation}
\label{gh6a}
\begin{aligned}
0&=f (e_{ii}\circ e_{jj}) = h(e_{ii})\circ e_{jj} + e_{ii}\circ g(e_{jj})\\
&=h_{1j}^{(ii)}e_{1j}+h_{2j}^{(ii)}e_{2j}+\dots+h_{jj}^{(ii)}e_{jj}+h_{jj}^{(ii)}e_{jj}+h_{j,j+1}^{(ii)}e_{j,j+1}+\dots+h_{jn}^{(ii)}e_{jn}\\
&+g_{ii}^{(jj)}e_{ii}+g_{i,i+1}^{(jj)}e_{i,i+1}+\dots+g_{in}^{(jj)}e_{in}+g_{1i}^{(jj)}e_{1i}+g_{2i}^{(jj)}e_{2i}+\dots+g_{ii}^{(jj)}e_{ii}.
\end{aligned}
\end{equation}

Equating the coefficients of $e_{1j},e_{2j},\dots,e_{i-1,j},e_{i+1,j},\dots,e_{j-1,j}$ respectively from both sides of \eqref{gh6},
\begin{equation}
\tag{1a} \label{gh1a}
\begin{aligned}
g_{1j}^{(ii)}=g_{2j}^{(ii)}=\dots=g_{i-1,j}^{(ii)}=g_{i+1,j}^{(ii)}=\dots=g_{j-1,j}^{(ii)}=0.
\end{aligned}
\end{equation}

Similarly, from \eqref{gh6},
\begin{equation}
\tag{1b} \label{gh1b}
\begin{aligned}
2g_{jj}^{(ii)}=2h_{ii}^{(jj)}=0 ~\text{(equating the coefficients of} ~e_{jj} ~\text{and} ~e_{ii})\implies g_{jj}^{(ii)}=h_{ii}^{(jj)}=0.
\end{aligned}
\end{equation}

Equating the coefficients of $e_{i,i+1},e_{i,i+2},\dots,e_{i,j-1},e_{i,j+1},\dots,e_{in}$ from \eqref{gh6},
\begin{equation}
\tag{1c} \label{gh1c}
\begin{aligned}
h_{i,i+1}^{(jj)}=h_{i,i+2}^{(jj)}=\dots =h_{i,j-1}^{(jj)} =h_{i,j+1}^{(jj)}=\dots=h_{in}^{(jj)}=0.
\end{aligned}
\end{equation}

Equating the coefficients of $e_{ij}$ from \eqref{gh6},
\begin{equation}
\tag{1d} \label{gh1d}
\begin{aligned}
g_{ij}^{(ii)}+h_{ij}^{(jj)}=0.
\end{aligned}
\end{equation}

Now, $f (e_{ii}e_{jj}) =0$  and
\begin{align*}
& g(e_{ii}) e_{jj} + e_{ii} h(e_{jj})\\
&=g_{1j}^{(ii)}e_{1j}+g_{2j}^{(ii)}e_{2j}+\dots+g_{jj}^{(ii)}e_{jj}+h_{ii}^{(jj)}e_{ii}+h_{i,i+1}^{(jj)}e_{i,i+1}+\dots+h_{in}^{(jj)}e_{in}\\
&=g_{1j}^{(ii)}e_{1j}+g_{2j}^{(ii)}e_{2j}+\dots+g_{i-1,j}^{(ii)}e_{i-1,j}+g_{i+1,j}^{(ii)}e_{i+1,j}+\dots+g_{j-1,j}^{(ii)}e_{j-1,j}\\
&+g_{jj}^{(ii)}e_{jj}+h_{ii}^{(jj)}e_{ii}\\
&+h_{i,i+1}^{(jj)}e_{i,i+1}+h_{i,i+2}^{(jj)}e_{i,i+2}+\dots+h_{i,j-1}^{(jj)}e_{i,j-1}+h_{i,j+1}^{(jj)}e_{i,j+1}+\dots+h_{in}^{(jj)}e_{in}\\
&+(g_{ij}^{(ii)}+h_{ij}^{(jj)})e_{ij}\\
&=0~\text{(by using \eqref{gh1a},\eqref{gh1b},\eqref{gh1c} and \eqref{gh1d})}.
\end{align*}

Similarly, by using \eqref{gh6a}, we get $h(e_{ii}) e_{jj} + e_{ii} g(e_{jj})=0$. Hence we prove that, $f (e_{ii}e_{jj}) =g(e_{ii}) e_{jj} + e_{ii} h(e_{jj})=h(e_{ii}) e_{jj} + e_{ii} g(e_{jj})$.
\end{proof}

\begin{lemma}
\label{lem4}
$f(e_{ii}e_{jk})=g(e_{ii})e_{jk}+e_{ii}h(e_{jk})=h(e_{ii})e_{jk}+e_{ii}g(e_{jk})$, for $j<k$.
\end{lemma}

\begin{proof}
Our aim is to prove the following.
\begin{equation}
\label{gh7}
f(e_{ii}e_{jk})=g(e_{ii})e_{jk}+e_{ii}h(e_{jk})=h(e_{ii})e_{jk}+e_{ii}g(e_{jk}), ~\text{where} ~j<k.
\end{equation}

\textbf{Case 1.} Let $i=j$. Then \eqref{gh7} equivalent to $f(e_{ij}e_{ii})=g(e_{ij})e_{ii}+e_{ij}h(e_{ii})=h(e_{ij})e_{ii}+e_{ij}g(e_{ii}), ~\text{where} ~i<j.$

For $i<j$,
\begin{equation}
\label{gh8}
\begin{aligned}
f(e_{ij})&=f(e_{ij}\circ e_{ii})=g(e_{ij})\circ e_{ii}+e_{ij}\circ h(e_{ii})\\
&=g_{1i}^{(ii)}e_{1j}+g_{2i}^{(ii)}e_{2j}+\dots +g_{ii}^{(ii)}e_{ij}+g_{jj}^{(ii)}e_{ij}+g_{j,j+1}^{(ii)}e_{i,j+1}+\dots+g_{jn}^{(ii)}e_{in}\\
&+h_{ii}^{(ij)}e_{ii}+h_{i,i+1}^{(ij)}e_{i,i+1}+\dots+h_{in}^{(ij)}e_{in}+h_{1i}^{(ij)}e_{1i}+h_{2i}^{(ij)}e_{2i}+\dots+h_{ii}^{(ij)}e_{ii},
\end{aligned}
\end{equation}

\begin{equation}
\label{gh9}
\begin{aligned}
f(e_{ij})&=f(e_{ij}\circ e_{ii})=h(e_{ij})\circ e_{ii}+e_{ij}\circ g(e_{ii})\\
&=h_{1i}^{(ii)}e_{1j}+h_{2i}^{(ii)}e_{2j}+\dots +h_{ii}^{(ii)}e_{ij}+h_{jj}^{(ii)}e_{ij}+h_{j,j+1}^{(ii)}e_{i,j+1}+\dots+h_{jn}^{(ii)}e_{in}\\
&+g_{ii}^{(ij)}e_{ii}+g_{i,i+1}^{(ij)}e_{i,i+1}+\dots+g_{in}^{(ij)}e_{in}+g_{1i}^{(ij)}e_{1i}+g_{2i}^{(ij)}e_{2i}+\dots+g_{ii}^{(ij)}e_{ii},
\end{aligned}
\end{equation}

\begin{equation}
\label{gh10}
\begin{aligned}
f(e_{ij})&=f(e_{ij}\circ e_{jj})=g(e_{ij})\circ e_{jj}+e_{ij}\circ h(e_{jj})\\
&=g_{1j}^{(ij)}e_{1j}+g_{2j}^{(ij)}e_{2j}+\dots+g_{jj}^{(ij)}e_{jj}+g_{jj}^{(ij)}e_{jj}+g_{j,j+1}^{(ij)}e_{j,j+1}+\dots+g_{jn}^{(ij)}e_{jn}\\
&+h_{jj}^{(jj)}e_{ij}+h_{j,j+1}^{(jj)}e_{i,j+1}+\dots+h_{jn}^{(jj)}e_{in}+h_{1i}^{(jj)}e_{1j}+h_{2i}^{(jj)}e_{2j}+\dots+h_{ii}^{(jj)}e_{ij},
\end{aligned}
\end{equation}

\begin{equation}
\label{gh11}
\begin{aligned}
f(e_{ij})&=f(e_{ij}\circ e_{jj})=h(e_{ij})\circ e_{jj}+e_{ij}\circ g(e_{jj})\\
&=h_{1j}^{(ij)}e_{1j}+h_{2j}^{(ij)}e_{2j}+\dots+h_{jj}^{(ij)}e_{jj}+h_{jj}^{(ij)}e_{jj}+h_{j,j+1}^{(ij)}e_{j,j+1}+\dots+h_{jn}^{(ij)}e_{jn}\\
&+g_{jj}^{(jj)}e_{ij}+g_{j,j+1}^{(jj)}e_{i,j+1}+\dots+g_{jn}^{(jj)}e_{in}+g_{1i}^{(jj)}e_{1j}+g_{2i}^{(jj)}e_{2j}+\dots+g_{ii}^{(jj)}e_{ij}.
\end{aligned}
\end{equation}

From \eqref{gh9} and \eqref{gh10}, equating the coefficients of $e_{1i},e_{2i},\dots,e_{i,i-1},e_{ii}$ respectively,
\begin{equation}
\tag{3a} \label{gh3a}
g_{1i}^{(ij)}=g_{2i}^{(ij)}=\dots=g_{i,i-1}^{(ij)}=2g_{ii}^{(ij)}=0\implies g_{1i}^{(ij)}=g_{2i}^{(ij)}=\dots=g_{i,i-1}^{(ij)}=g_{ii}^{(ij)}=0.
\end{equation}

From \eqref{gh8} and \eqref{gh11}, equating the coefficients of $e_{1i},e_{2i},\dots,e_{i,i-1},e_{ii}$ respectively,
\begin{equation}
\tag{3b} \label{gh3b}
h_{1i}^{(ij)}=h_{2i}^{(ij)}=\dots=h_{i,i-1}^{(ij)}=2h_{ii}^{(ij)}=0\implies h_{1i}^{(ij)}=h_{2i}^{(ij)}=\dots=h_{i,i-1}^{(ij)}=h_{ii}^{(ij)}=0.
\end{equation}

From \eqref{gh6a}, equating the coefficients of $e_{jj},e_{ii}$ respectively,
\begin{equation}
\tag{2a} \label{gh2a}
2h_{jj}^{(ii)}=2g_{ii}^{(jj)}=0\implies h_{jj}^{(ii)}=g_{ii}^{(jj)}=0.
\end{equation}

From \eqref{gh6a}, equating the coefficients of $e_{j,j+1},e_{j,j+2},\dots,e_{jn}$ respectively,
\begin{equation}
\tag{2b} \label{gh2b}
h_{j,j+1}^{(ii)}=h_{j,j+2}^{(ii)}=\dots=h_{jn}^{(ii)}=0.
\end{equation}

From \eqref{gh6}, equating the coefficients of $e_{j,j+1},e_{j,j+2},\dots,e_{jn}$ respectively,
\begin{equation}
\tag{1e} \label{gh1e}
g_{j,j+1}^{(ii)}=g_{j,j+2}^{(ii)}=\dots=g_{jn}^{(ii)}=0.
\end{equation}

Now,\begin{align*}
&f(e_{ij}e_{ii})=0.\\
&g(e_{ij})e_{ii}+e_{ij}h(e_{ii})\\
&=g_{1i}^{(ij)}e_{1i}+g_{2i}^{(ij)}e_{2i}+\dots+g_{ii}^{(ij)}e_{ii}+h_{jj}^{(ii)}e_{ij}+h_{j,j+1}^{(ii)}e_{i,j+1}+\dots+h_{jn}^{(ii)}e_{in}
=0\\
&~\text{(by \eqref{gh3a},\eqref{gh2a},\eqref{gh2b})}.\\
&h(e_{ij})e_{ii}+e_{ij}g(e_{ii})\\
&=h_{1i}^{(ij)}e_{1i}+h_{2i}^{(ij)}e_{2i}+\dots+h_{ii}^{(ij)}e_{ii}+g_{jj}^{(ii)}e_{ij}+g_{j,j+1}^{(ii)}e_{i,j+1}+\dots+g_{jn}^{(ii)}e_{in}
=0\\
&~\text{(by \eqref{gh3b},\eqref{gh1b},\eqref{gh1e})}.
\end{align*}

\textbf{Case 2.} Claim: $f(e_{ii}e_{jk})=g(e_{ii})e_{jk}+e_{ii}h(e_{jk})=h(e_{ii})e_{jk}+e_{ii}g(e_{jk}), ~\text{where} ~i<j.$

\begin{equation}
\label{gh12}
\begin{aligned}
0&=f(e_{ii}\circ e_{jk})=g(e_{ii})\circ e_{jk}+e_{ii}\circ h(e_{jk})\\
&=g_{1j}^{(ii)}e_{1k}+g_{2j}^{(ii)}e_{2k}+\dots+g_{jj}^{(ii)}e_{jk}+g_{kk}^{(ii)}e_{jk}+g_{k,k+1}^{(ii)}e_{j,k+1}+\dots+g_{kn}^{(ii)}e_{jn}\\
&+h_{ii}^{(jk)}e_{ii}+h_{i,i+1}^{(jk)}e_{i,i+1}+\dots+h_{in}^{(jk)}e_{in}+h_{1i}^{(jk)}e_{1i}+h_{2i}^{(jk)}e_{2i}+\dots+h_{ii}^{(jk)}e_{ii}.
\end{aligned}
\end{equation}

Equating the coefficients of $e_{ik},e_{ii}, e_{i,i+1},\dots,e_{i,k-1},e_{i,k+1},\dots,e_{in}$ respectively from both sides of \eqref{gh12},
\begin{equation}
\label{gh13}
\begin{aligned}
&g_{ij}^{(ii)}+h_{ik}^{(jk)}=0\\
&2h_{ii}^{(jk)}=h_{i,i+1}^{(jk)}=\dots=h_{i,k-1}^{(jk)}=h_{i,k+1}^{(jk)}=\dots=h_{in}^{(jk)}=0\\
&\implies h_{ii}^{(jk)}=h_{i,i+1}^{(jk)}=\dots=h_{i,k-1}^{(jk)}=h_{i,k+1}^{(jk)}=\dots=h_{in}^{(jk)}=0.
\end{aligned}
\end{equation}

\begin{align*}
&f(e_{ii}e_{jk})=0.\\
&g(e_{ii})e_{jk}+e_{ii}h(e_{jk})\\
&=g_{1j}^{(ii)}e_{1k}+g_{2j}^{(ii)}e_{2k}+\dots+g_{jj}^{(ii)}e_{jk}+h_{ii}^{(jk)}e_{ii}+h_{i,i+1}^{(jk)}e_{i,i+1}+\dots+h_{in}^{(jk)}e_{in}\\
&=g_{1j}^{(ii)}e_{1k}+g_{2j}^{(ii)}e_{2k}+\dots+g_{i-1,j}^{(ii)}e_{i-1,k}+g_{i+1,j}^{(ii)}e_{i+1,k}+\dots+g_{j-1,j}^{(ii)}e_{j-1,k}\\
&+g_{jj}^{(ii)}e_{jk}\\
&+(g_{ij}^{(ii)}+h_{ik}^{(jk)})e_{ik}\\
&+h_{ii}^{(jk)}e_{ii}+h_{i,i+1}^{(jk)}e_{i,i+1}+\dots+h_{i,k-1}^{(jk)}e_{i,k-1}+h_{i,k+1}^{(jk)}e_{i,k+1}+\dots+h_{in}^{(jk)}e_{in}\\
&=0~\text{( by \eqref{gh1a},\eqref{gh1b} and \eqref{gh13})}.\\
&~\text{Similarly, using}~ f(e_{ii}\circ e_{jk})=h(e_{ii})\circ e_{jk}+e_{ii}\circ g(e_{jk}),\\
&~\text{we have} ~h(e_{ii})e_{jk}+e_{ii}g(e_{jk})=0.
\end{align*}

\textbf{Case 3.} Claim: \begin{equation}
\label{gh14}
f(e_{ii}e_{jk})=g(e_{ii})e_{jk}+e_{ii}h(e_{jk})=h(e_{ii})e_{jk}+e_{ii}g(e_{jk}), ~\text{where} ~i>j.
\end{equation}

\textbf{Subcase 1.} Let $k=i$. Then \eqref{gh14} is equivalent to $f(e_{jj}e_{ij})=g(e_{jj})e_{ij}+e_{jj}h(e_{ij})=h(e_{jj})e_{ij}+e_{jj}g(e_{ij})$, for $i<j$.

Equating the coefficients of $e_{1i},e_{2i},\dots,e_{i-1,i},e_{ii}$ from \eqref{gh6a},
\begin{equation}
\tag{2c} \label{gh2c}
g_{1i}^{(jj)}=g_{2i}^{(jj)}=\dots=g_{i-1,i}^{(jj)}=2g_{ii}^{(jj)}=0 \implies g_{1i}^{(jj)}=g_{2i}^{(jj)}=\dots=g_{i-1,i}^{(jj)}=g_{ii}^{(jj)}=0.
\end{equation}

Comparing the coefficients of $e_{jj},e_{j,j+1},\dots,e_{jn}$ from \eqref{gh9} and \eqref{gh11},
\begin{equation}
\tag{3c} \label{gh3c}
h_{jj}^{(ij)}=h_{j,j+1}^{(ij)}=\dots=h_{jn}^{(ij)}=0.
\end{equation}

\begin{align*}
&f(e_{jj}e_{ij})=0.\\
&g(e_{jj})e_{ij}+e_{jj}h(e_{ij})\\
&=g_{1i}^{(jj)}e_{1j}+g_{2i}^{(jj)}e_{2j}+\dots+g_{ii}^{(jj)}e_{ij}+h_{jj}^{(ij)}e_{jj}+h_{j,j+1}^{(ij)}e_{j,j+1}+\dots+h_{jn}^{(ij)}e_{jn}\\
&=0~\text{(by \eqref{gh2c} and \eqref{gh3c})}.\\
&\text{In a similar way,}~h(e_{jj})e_{ij}+e_{jj}g(e_{ij})=0.
\end{align*}

\textbf{Subcase 2.} Let $k\neq i$.

\begin{equation}
\label{gh15}
\begin{aligned}
0&=f(e_{ii}\circ e_{jk})\\
&=g(e_{ii})\circ e_{jk}+e_{ii}\circ h(e_{jk})\\
&=g_{1j}^{(ii)}e_{1k}+g_{2j}^{(ii)}e_{2k}+\dots+g_{jj}^{(ii)}e_{jk}+g_{kk}^{(ii)}e_{jk}+g_{k,k+1}^{(ii)}e_{j,k+1}+\dots+g_{kn}^{(ii)}e_{jn}\\
&+h_{ii}^{(jk)}e_{ii}+h_{i,i+1}^{(jk)}e_{i,i+1}+\dots+h_{in}^{(jk)}e_{in}+h_{1i}^{(jk)}e_{1i}+h_{2i}^{(jk)}e_{2i}+\dots+h_{ii}^{(jk)}e_{ii}.
\end{aligned}
\end{equation}

Equating the coefficients of $e_{ii},e_{i,i+1},\dots,e_{in}$ from \eqref{gh15},
\begin{equation}
\label{gh16}
2h_{ii}^{(jk)}=h_{i,i+1}^{(jk)}=\dots=h_{in}^{(jk)}=0\implies h_{ii}^{(jk)}=h_{i,i+1}^{(jk)}=\dots=h_{in}^{(jk)}=0.
\end{equation}

Since \eqref{gh6} is also true for $i>j$, equating the coefficients of $e_{1j},e_{2j},\dots,e_{jj}$ from \eqref{gh6},
\begin{equation}
\tag{1f} \label{gh1f}
g_{1j}^{(ii)}=g_{2j}^{(ii)}=\dots=g_{jj}^{(ii)}=0.
\end{equation}

\begin{align*}
&f(e_{ii} e_{jk})=0.\\
&g(e_{ii}) e_{jk}+e_{ii} h(e_{jk})\\
&=g_{1j}^{(ii)}e_{1k}+g_{2j}^{(ii)}e_{2k}+\dots+g_{jj}^{(ii)}e_{jk}+h_{ii}^{(jk)}e_{ii}+h_{i,i+1}^{(jk)}e_{i,i+1}+\dots+h_{in}^{(jk)}e_{in}=0\\
&~\text{(by \eqref{gh1f} and \eqref{gh16} )}.\\
&~\text{In a similar fashion,} ~h(e_{ii})e_{jk}+e_{ii}g(e_{jk})=0.
\end{align*}
\end{proof}

\begin{lemma}
\label{lem5}
$f(e_{ij}e_{kl})=g(e_{ij})e_{kl}+e_{ij}h(e_{kl})=h(e_{ij})e_{kl}+e_{ij}g(e_{kl})$, for $i<j,~k<l$.
\end{lemma}

\begin{proof}
We prove
\begin{equation}
\label{gh17}
f(e_{ij}e_{kl})=g(e_{ij})e_{kl}+e_{ij}h(e_{kl})=h(e_{ij})e_{kl}+e_{ij}g(e_{kl}), ~\text{where} ~i<j,~k<l.
\end{equation}

\textbf{Case 1.} Let $j=k$. Then \eqref{gh17} equivalent to $f(e_{kj}e_{ik})=g(e_{kj})e_{ik}+e_{kj}h(e_{ik})=h(e_{kj})e_{ik}+e_{kj}g(e_{ik}), ~\text{where} ~i<k<j.$

For $i<k<j$,
\begin{equation}
\label{gh18}
\begin{aligned}
f(e_{ij})&=f(e_{kj}\circ e_{ik})=g(e_{kj})\circ e_{ik}+e_{kj}\circ h(e_{ik})\\
&=g_{1i}^{(kj)}e_{1k}+g_{2i}^{(kj)}e_{2k}+\dots+g_{ii}^{(kj)}e_{ik}+g_{kk}^{(kj)}e_{ik}+g_{k,k+1}^{(kj)}e_{i,k+1}+\dots+g_{kn}^{(kj)}e_{in}\\
&+h_{jj}^{(ik)}e_{kj}+h_{j,j+1}^{(ik)}e_{k,j+1}+\dots+h_{jn}^{(ik)}e_{kn}+h_{1k}^{(ik)}e_{1j}+h_{2k}^{(ik)}e_{2j}+\dots+h_{kk}^{(ik)}e_{kj}.
\end{aligned}
\end{equation}

Comparing the coefficients of $e_{1k},e_{2k},\dots,e_{i-1,k}$ from \eqref{gh8} and \eqref{gh18},
\begin{equation}
\tag{3d}\label{gh3d}
g_{1i}^{(kj)}=g_{2i}^{(kj)}=\dots=g_{i-1,i}^{(kj)}=0.
\end{equation}

Comparing the coefficients of $e_{k,j+1},e_{k,j+2},\dots,e_{kn}$ from \eqref{gh8} and \eqref{gh18},
\begin{equation}
\tag{3e}\label{gh3e}
h_{j,j+1}^{(ik)}=h_{j,j+2}^{(ik)}=\dots=h_{jn}^{(ik)}=0.
\end{equation}

As discussed in Lemma \eqref{lem4}, from $h(e_{ii})e_{kj}+e_{ii}g(e_{kj})=f(e_{ii}e_{kj})=0$ and $h(e_{ik})e_{jj}+e_{ik}g(e_{jj})=f(e_{ik}e_{jj})=0$, equating the coefficients of $e_{ii}$ and $e_{jj}$ respectively,
\begin{equation}
\label{gh19}
g_{ii}^{(kj)}=0=h_{jj}^{(ik)}.
\end{equation}

\begin{align*}
&f(e_{kj} e_{ik})=0.\\
&g(e_{kj})e_{ik}+e_{kj} h(e_{ik})\\
&=g_{1i}^{(kj)}e_{1k}+g_{2i}^{(kj)}e_{2k}+\dots+g_{ii}^{(kj)}e_{ik} +h_{jj}^{(ik)}e_{kj}+h_{j,j+1}^{(ik)}e_{k,j+1}+\dots+h_{jn}^{(ik)}e_{kn}\\
&=g_{1i}^{(kj)}e_{1k}+g_{2i}^{(kj)}e_{2k}+\dots+g_{i-1,i}^{(kj)}e_{i-1,k}\\
&+g_{ii}^{(kj)}e_{ik}+h_{jj}^{(ik)}e_{kj}\\
&+h_{j,j+1}^{(ik)}e_{k,j+1}+h_{j,j+2}^{(ik)}e_{k,j+2}+\dots+h_{jn}^{(ik)}e_{kn}\\
&=0~\text{(by \eqref{gh3d},\eqref{gh19},\eqref{gh3e}) }.\\
&~\text{Similarly, we can prove that,}~h(e_{kj})e_{ik}+e_{kj} g(e_{ik})=0.
\end{align*}

\textbf{Case 2.} Let $j<k$.

For $i<j<k<l$,
\begin{equation}
\label{gh20}
\begin{aligned}
0&=f(e_{ij}\circ e_{kl})=g(e_{ij})\circ e_{kl}+e_{ij}\circ h(e_{kl})\\
&=g_{1k}^{(ij)}e_{1l}+g_{2k}^{(ij)}e_{2l}+\dots+g_{kk}^{(ij)}e_{kl}+g_{ll}^{(ij)}e_{kl}+g_{l,l+1}^{(ij)}e_{k,l+1}+\dots+g_{ln}^{(ij)}e_{kn}\\
&+h_{jj}^{(kl)}e_{ij}+h_{j,j+1}^{(kl)}e_{i,j+1}+\dots+h_{jn}^{(kl)}e_{in}+h_{1i}^{(kl)}e_{1j}+h_{2i}^{(kl)}e_{2j}+\dots+h_{ii}^{(kl)}e_{ij}.
\end{aligned}
\end{equation}

Equating coefficients of $e_{1l},e_{2l},\dots,e_{i-1,l},e_{i+1,l},\dots,e_{k-1,l}$ from \eqref{gh20},
\begin{equation}
\tag{20a}\label{20a}
g_{1k}^{(ij)}=g_{2k}^{(ij)}=\dots=g_{i-1,k}^{(ij)}=g_{i+1,k}^{(ij)}=\dots=g_{k-1,k}^{(ij)}=0.
\end{equation}

Equating coefficients of $e_{il}$ from \eqref{gh20},
\begin{equation}
\tag{20b}\label{20b}
g_{ik}^{(ij)}+h_{jl}^{(kl)}=0.
\end{equation}

Equating coefficients of $e_{i,j+1},e_{i,j+2},\dots,e_{i,l-1},e_{i,l+1},\dots,e_{in}$ from \eqref{gh20},
\begin{equation}
\tag{20c}\label{20c}
h_{j,j+1}^{(kl)}=h_{j,j+2}^{(kl)}=\dots=h_{j,l-1}^{(kl)}=h_{j,l+1}^{(kl)}=\dots=h_{jn}^{(kl)}=0.
\end{equation}

By Lemma \eqref{lem4}, equating the coefficients of $e_{kk}$ and $e_{jj}$ from $0=f(e_{ij}e_{kk})=g(e_{ij})e_{kk}+e_{ij}h(e_{kk})$ and $0=f(e_{jj}e_{kl})=g(e_{jj})e_{kl}+e_{jj}h(e_{kl})$ respectively,
\begin{equation}
\label{gh21}
g_{kk}^{(ij)}=0=g_{jj}^{(kl)}.
\end{equation}

\begin{align*}
&f(e_{ij} e_{kl})=0.\\
&g(e_{ij})e_{kl}+e_{ij} h(e_{kl})\\
&=g_{1k}^{(ij)}e_{1l}+g_{2k}^{(ij)}e_{2l}+\dots+g_{kk}^{(ij)}e_{kl}+h_{jj}^{(kl)}e_{ij}+h_{j,j+1}^{(kl)}e_{i,j+1}+\dots+h_{jn}^{(kl)}e_{in}\\
&=g_{1k}^{(ij)}e_{1l}+g_{2k}^{(ij)}e_{2l}+\dots+g_{i-1,k}^{(ij)}e_{i-1,l}+g_{i+1,k}^{(ij)}e_{i+1,l}+\dots+g_{k-1,k}^{(ij)}e_{k-1,l}\\
&+(g_{ik}^{(ij)}+h_{jl}^{(kl)})e_{il}+g_{kk}^{(ij)}e_{kl}++h_{jj}^{(kl)}e_{ij}\\
&+h_{j,j+1}^{(kl)}e_{i,j+1}+h_{j,j+2}^{(kl)}e_{i,j+2}+\dots+h_{j,l-1}^{(kl)}e_{i,l-1}+h_{j,l+1}^{(kl)}e_{i,l+1}\dots+h_{jn}^{(kl)}e_{in}\\
&=0~\text{(by \eqref{20a},\eqref{20b},\eqref{gh21},\eqref{20c})}.\\
&~\text{Simiarly we can show that,}~ h(e_{ij})e_{kl}+e_{ij} g(e_{kl})=0.
\end{align*}

\textbf{Case 3.} Let $j>k$. So we have $i<j,j>k,k<l$.

From Lemma \eqref{lem4},
\begin{equation}
\label{gh22}
\begin{aligned}
&0=f(e_{ij}e_{kk})=g(e_{ij})e_{kk}+e_{ij}h(e_{kk})\\
&=g_{1k}^{(ij)}e_{1k}+g_{2k}^{(ij)}e_{2k}+\dots+g_{kk}^{(ij)}e_{kk}+h_{jj}^{(kk)}e_{ij}+h_{j,j+1}^{(kk)}e_{i,j+1}+\dots+h_{jn}^{(kk)}e_{in}\\
&\implies g_{1k}^{(ij)}=g_{2k}^{(ij)}=\dots=g_{kk}^{(ij)}=0.
\end{aligned}
\end{equation}

Similarly,
\begin{equation}
\label{gh23}
\begin{aligned}
&0=f(e_{jj}e_{kl})=g(e_{jj})e_{kl}+e_{jj}h(e_{kl})\\
&=g_{1k}^{(jj)}e_{1l}+g_{2k}^{(jj)}e_{2l}+\dots+g_{kk}^{(jj)}e_{kl}+h_{jj}^{(kl)}e_{jj}+h_{j,j+1}^{(kl)}e_{j,j+1}+\dots+h_{jn}^{(kl)}e_{jn}\\
&\implies h_{jj}^{(kl)}=h_{j,j+1}^{(kl)}=\dots=h_{jn}^{(kl)}=0.
\end{aligned}
\end{equation}

\begin{align*}
&f(e_{ij} e_{kl})=0.\\
&g(e_{ij})e_{kl}+e_{ij} h(e_{kl})\\
&=g_{1k}^{(ij)}e_{1l}+g_{2k}^{(ij)}e_{2l}+\dots+g_{kk}^{(ij)}e_{kl}+h_{jj}^{(kl)}e_{ij}+h_{j,j+1}^{(kl)}e_{i,j+1}+\dots+h_{jn}^{(kl)}e_{in}=0\\
&~\text{(by \eqref{gh22} and \eqref{gh23})}.\\
&~\text{Similarly we can show that}~ h(e_{ij})e_{kl}+e_{ij} g(e_{kl})=0.
\end{align*}
\end{proof}

\textbf{Proof of Theorem \eqref{thm1}:}
Let $f$ be a Jordan $\{g, h\}$-derivation on $\mathcal{T}_n(C)$. Since $f(e_{ii}^2)=f(e_{ii})=g(e_{ii})e_{ii}+e_{ii}h(e_{ii})$, we get $f(e_{ii}^2)=h(e_{ii})e_{ii}+e_{ii}g(e_{ii})$, for all  $i=1,2,\dots,n$ (by Lemma \ref{pro1}). By Lemma \ref{lem3}-\ref{lem5}, we can prove that $f (xy) = g(x)y + xh(y) = h(x)y + xg(y) ~\text{for all}~ x, y \in \mathcal{T}_n(C)$.

\begin{remark}
Now we give an example where $f$ is a Jordan $\{g,h\}$-derivation over $\mathcal{T}_2(C)$, but not a $\{g,h\}$-derivation over $\mathcal{T}_2(C)$. For example, let $g:\mathcal{T}_2(C)\rightarrow \mathcal{T}_2(C)$ as $g(x)=a\circ x$, where $a=e_{11}+e_{12}+e_{22}$ and $C$ is the field of complex numbers. Then $0$ is Jordan $\{g, -g\}$ derivation, but $0(e_{11}^2)=0 \neq -e_{12}=g(e_{11})e_{11}+e_{11}((-g)(e_{11}))$, that is, $0$ is not a $\{g, -g\}$ derivation.
\end{remark}

\section{Jordan \{g,h\} derivation on $\mathcal{M}_n(C)$}
Now we study Jordan $\{g,h\}$-derivation over full matrix algebra $\mathcal{M}_n(C)$. We prove every Jordan $\{g,h\}$-derivation on $\mathcal{M}_n(C)$ is a $\{g,h\}$-derivation.

\begin{theorem}
\label{thm2}
Let $\mathcal{M}_n(C)$ be the algebra of $n\times n$ matrices over $C$. Then every Jordan $\{g,h\}$-derivation on $\mathcal{M}_n(C)$, $n \geq 2$, into itself is a $\{g,h\}$-derivation.
\end{theorem}

Let $f:\mathcal{M}_n(C)\rightarrow \mathcal{M}_n(C)$ be a Jordan $\{g, h\}$-derivation.

First we prove
\begin{equation}
\label{gh1z}
f (e_{ij}e_{kl}) = g(e_{ij})e_{kl} + e_{ij}h(e_{kl}) = h(e_{ij})e_{kl} + e_{ij}g(e_{kl}),
\end{equation}
which is equivalent to
\begin{equation}
\label{gh2z}
f (e_{kl}e_{ij}) = g(e_{kl})e_{ij} + e_{kl}h(e_{ij}) = h(e_{kl})e_{ij} + e_{kl}g(e_{ij})~\text{(by Lemma \ref{pro2})}.
\end{equation}

Now, let
\begin{equation}
\label{gh3z}
g(e_{ij})=\sum_{k=1}^{n}\sum_{l=1}^{n}g_{kl}^{(ij)}e_{kl}~, ~\text{where} ~g_{kl}^{(ij)} \in C
\end{equation}

and \begin{equation}
\label{gh4z}
h(e_{ij})=\sum_{k=1}^{n}\sum_{l=1}^{n}h_{kl}^{(ij)}e_{kl}~, ~\text{where} ~h_{kl}^{(ij)} \in C.
\end{equation}

Now we use \eqref{gh3z} and \eqref{gh4z} to derive next few identities. Let $i\neq j$.
\begin{equation}
\label{gh1y}
\begin{aligned}
0&=f(e_{ii}\circ e_{jj})=g(e_{ii})\circ e_{jj}+e_{ii}\circ h(e_{jj})\\
&=g_{1j}^{(ii)}e_{1j}+g_{2j}^{(ii)}e_{2j}+\dots+g_{nj}^{(ii)}e_{nj}+g_{j1}^{(ii)}e_{j1}+g_{j2}^{(ii)}e_{j2}+\dots+g_{jn}^{(ii)}e_{jn}\\
&+h_{i1}^{(jj)}e_{i1}+h_{i2}^{(jj)}e_{i2}+\dots+h_{in}^{(jj)}e_{in}+h_{1i}^{(jj)}e_{1i}+h_{2i}^{(jj)}e_{2i}+\dots+h_{ni}^{(jj)}e_{ni}.
\end{aligned}
\end{equation}

\begin{equation}
\label{gh2y}
\begin{aligned}
0&=f(e_{ii}\circ e_{jj})=h(e_{ii})\circ e_{jj}+e_{ii}\circ g(e_{jj})\\
&=h_{1j}^{(ii)}e_{1j}+h_{2j}^{(ii)}e_{2j}+\dots+h_{nj}^{(ii)}e_{nj}+h_{j1}^{(ii)}e_{j1}+h_{j2}^{(ii)}e_{j2}+\dots+h_{jn}^{(ii)}e_{jn}\\
&+g_{i1}^{(jj)}e_{i1}+g_{i2}^{(jj)}e_{i2}+\dots+g_{in}^{(jj)}e_{in}+g_{1i}^{(jj)}e_{1i}+g_{2i}^{(jj)}e_{2i}+\dots+g_{ni}^{(jj)}e_{ni}.
\end{aligned}
\end{equation}

\begin{equation}
\label{gh3y}
\begin{aligned}
f(e_{ij})&=f(e_{ii}\circ e_{ij})=g(e_{ii})\circ e_{ij}+e_{ii}\circ h(e_{ij})\\
&=g_{1i}^{(ii)}e_{1j}+g_{2i}^{(ii)}e_{2j}+\dots+g_{ni}^{(ii)}e_{nj}+g_{j1}^{(ii)}e_{i1}+g_{j2}^{(ii)}e_{i2}+\dots+g_{jn}^{(ii)}e_{in}\\
&+h_{i1}^{(ij)}e_{i1}+h_{i2}^{(ij)}e_{i2}+\dots+h_{in}^{(ij)}e_{in}+h_{1i}^{(ij)}e_{1i}+h_{2i}^{(ij)}e_{2i}+\dots+h_{ni}^{(ij)}e_{ni}.
\end{aligned}
\end{equation}

\begin{equation}
\label{gh4y}
\begin{aligned}
f(e_{ij})&=f(e_{ii}\circ e_{ij})=h(e_{ii})\circ e_{ij}+e_{ii}\circ g(e_{ij})\\
&=h_{1i}^{(ii)}e_{1j}+h_{2i}^{(ii)}e_{2j}+\dots+h_{ni}^{(ii)}e_{nj}+h_{j1}^{(ii)}e_{i1}+h_{j2}^{(ii)}e_{i2}+\dots+h_{jn}^{(ii)}e_{in}\\
&+g_{i1}^{(ij)}e_{i1}+g_{i2}^{(ij)}e_{i2}+\dots+g_{in}^{(ij)}e_{in}+g_{1i}^{(ij)}e_{1i}+g_{2i}^{(ij)}e_{2i}+\dots+g_{ni}^{(ij)}e_{ni}.
\end{aligned}
\end{equation}

\begin{equation}
\label{gh5y}
\begin{aligned}
f(e_{ij})&=f(e_{ij}\circ e_{jj})=g(e_{ij})\circ e_{jj}+e_{ij}\circ h(e_{jj})\\
&=g_{1j}^{(ij)}e_{1j}+g_{2j}^{(ij)}e_{2j}+\dots+g_{nj}^{(ij)}e_{nj}+g_{j1}^{(ij)}e_{j1}+g_{j2}^{(ij)}e_{j2}+\dots+g_{jn}^{(ij)}e_{jn}\\
&+h_{j1}^{(jj)}e_{i1}+h_{j2}^{(jj)}e_{i2}+\dots+h_{jn}^{(jj)}e_{in}+h_{1i}^{(jj)}e_{1j}+h_{2i}^{(jj)}e_{2j}+\dots+h_{ni}^{(jj)}e_{nj}.
\end{aligned}
\end{equation}

\begin{equation}
\label{gh6y}
\begin{aligned}
f(e_{ij})&=f(e_{ij}\circ e_{jj})=h(e_{ij})\circ e_{jj}+e_{ij}\circ g(e_{jj})\\
&=h_{1j}^{(ij)}e_{1j}+h_{2j}^{(ij)}e_{2j}+\dots+h_{nj}^{(ij)}e_{nj}+h_{j1}^{(ij)}e_{j1}+h_{j2}^{(ij)}e_{j2}+\dots+h_{jn}^{(ij)}e_{jn}\\
&+g_{j1}^{(jj)}e_{i1}+g_{j2}^{(jj)}e_{i2}+\dots+g_{jn}^{(jj)}e_{in}+g_{1i}^{(jj)}e_{1j}+g_{2i}^{(jj)}e_{2j}+\dots+g_{ni}^{(jj)}e_{nj}.
\end{aligned}
\end{equation}

Similarly, we have

\begin{equation}
\label{gh7y}
\begin{aligned}
f(e_{ji})&=g_{1j}^{(jj)}e_{1i}+g_{2j}^{(jj)}e_{2i}+\dots+g_{nj}^{(jj)}e_{ni}+g_{i1}^{(jj)}e_{j1}+g_{i2}^{(jj)}e_{j2}+\dots+g_{in}^{(jj)}e_{in}\\
&+h_{j1}^{(ji)}e_{j1}+h_{j2}^{(ji)}e_{j2}+\dots+h_{jn}^{(ji)}e_{jn}+h_{1j}^{(ji)}e_{1j}+h_{2j}^{(ji)}e_{2j}+\dots+h_{nj}^{(ji)}e_{nj},
\end{aligned}
\end{equation}

\begin{equation}
\label{gh8y}
\begin{aligned}
f(e_{ji})&=h_{1j}^{(jj)}e_{1i}+h_{2j}^{(jj)}e_{2i}+\dots+h_{nj}^{(jj)}e_{ni}+h_{i1}^{(jj)}e_{j1}+h_{i2}^{(jj)}e_{j2}+\dots+h_{in}^{(jj)}e_{in}\\
&+g_{j1}^{(ji)}e_{j1}+g_{j2}^{(ji)}e_{j2}+\dots+g_{jn}^{(ji)}e_{jn}+g_{1j}^{(ji)}e_{1j}+g_{2j}^{(ji)}e_{2j}+\dots+g_{nj}^{(ji)}e_{nj},
\end{aligned}
\end{equation}

\begin{equation}
\label{gh9y}
\begin{aligned}
f(e_{ji})&=g_{1i}^{(ji)}e_{1i}+g_{2i}^{(ji)}e_{2i}+\dots+g_{ni}^{(ji)}e_{ni}+g_{i1}^{(ji)}e_{i1}+g_{i2}^{(ji)}e_{i2}+\dots+g_{in}^{(ji)}e_{in}\\
&+h_{i1}^{(ii)}e_{j1}+h_{i2}^{(ii)}e_{j2}+\dots+h_{in}^{(ii)}e_{jn}+h_{1j}^{(ii)}e_{1i}+h_{2j}^{(ii)}e_{2i}+\dots+h_{nj}^{(ii)}e_{ni},
\end{aligned}
\end{equation}

\begin{equation}
\label{gh10y}
\begin{aligned}
f(e_{ji})&=h_{1i}^{(ji)}e_{1i}+h_{2i}^{(ji)}e_{2i}+\dots+h_{ni}^{(ji)}e_{ni}+h_{i1}^{(ji)}e_{i1}+h_{i2}^{(ji)}e_{i2}+\dots+h_{in}^{(ji)}e_{in}\\
&+g_{i1}^{(ii)}e_{j1}+g_{i2}^{(ii)}e_{j2}+\dots+g_{in}^{(ii)}e_{jn}+g_{1j}^{(ii)}e_{1i}+g_{2j}^{(ii)}e_{2i}+\dots+g_{nj}^{(ii)}e_{ni}.
\end{aligned}
\end{equation}

We have several lemmas.
\begin{lemma}
\label{lem5z}
$f (e_{ii}e_{ii}) =g(e_{ii}) e_{ii} + e_{ii} h(e_{ii})=h(e_{ii}) e_{ii} + e_{ii} g(e_{ii})$, for $i=1,2,\dots,n$.
\end{lemma}

\begin{proof}
Equating the coefficients of $e_{1j},\dots,e_{i-1,j},e_{i+1,j},\dots,e_{nj}$ from \eqref{gh3y} and \eqref{gh4y}, we have
\begin{equation}
\label{gh11y}
g_{1i}^{(ii)}=h_{1i}^{(ii)},\dots,g_{i-1,i}^{(ii)}=h_{i-1,i}^{(ii)},g_{i+1,i}^{(ii)}=h_{i+1,i}^{(ii)},\dots,g_{ni}^{(ii)}=h_{ni}^{(ii)}.
\end{equation}

Similarly, equating the coefficients of $e_{j1},\dots,e_{j,i-1},e_{j,i+1},\dots,e_{jn}$ from \eqref{gh9y} and \eqref{gh10y}, we have

\begin{equation}
\label{gh12y}
g_{i1}^{(ii)}=h_{i1}^{(ii)},\dots,g_{i,i-1}^{(ii)}=h_{i,i-1}^{(ii)},g_{i,i+1}^{(ii)}=h_{i,i+1}^{(ii)},\dots,g_{in}^{(ii)}=h_{in}^{(ii)}.
\end{equation}

By using \eqref{gh11y} and \eqref{gh12y}, we have
\begin{align*}
f(e_{ii}\circ e_{ii})=2(g(e_{ii})e_{ii}+e_{ii}h(e_{ii})) \implies f(e_{ii}^2)=g(e_{ii})e_{ii}+e_{ii}h(e_{ii}).
\end{align*}

Similarly we can prove that,
\begin{align*}
f(e_{ii}^2)=h(e_{ii})e_{ii}+e_{ii}g(e_{ii}).
\end{align*}

\end{proof}

\begin{lemma}
\label{lem6z}
$f (e_{ii}e_{jj}) =g(e_{ii}) e_{jj} + e_{ii} h(e_{jj})=h(e_{ii}) e_{jj} + e_{ii} g(e_{jj})$, for $i\neq j$.
\end{lemma}

\begin{proof}
Comparing the coefficients of $e_{1j},\dots,e_{i-1,j},e_{i+1,j},\dots,e_{nj}$ from \eqref{gh1y},

\begin{equation}
\label{gh13y}
g_{1j}^{(ii)}=\dots=g_{i-1,j}^{(ii)}=g_{i+1,j}^{(ii)}=\dots=g_{nj}^{(ii)}=0.
\end{equation}

Comparing the coefficients of $e_{ij}$ from \eqref{gh1y},

\begin{equation}
\label{gh14y}
g_{ij}^{(ii)}+h_{ij}^{(jj)}=0.
\end{equation}

Comparing the coefficients of $e_{i1},\dots,e_{i,j-1},e_{i,j+1},\dots,e_{in}$ from \eqref{gh1y},

\begin{equation}
\label{gh15y}
h_{i1}^{(jj)}=\dots=h_{i,j-1}^{(jj)}=h_{i,j+1}^{(jj)}=\dots=h_{in}^{(jj)}=0.
\end{equation}

Using \eqref{gh13y},\eqref{gh14y} and \eqref{gh15y},

\begin{align*}
f (e_{ii}e_{jj})=0=g(e_{ii}) e_{jj} + e_{ii} h(e_{jj}).
\end{align*}

Similarly by using \eqref{gh2y},
\begin{align*}
f (e_{ii}e_{jj})=0=h(e_{ii}) e_{jj} + e_{ii} g(e_{jj}).
\end{align*}

\end{proof}

\begin{lemma}
\label{lem7z}
$f (e_{ii}e_{jk}) =g(e_{ii}) e_{jk} + e_{ii} h(e_{jk})=h(e_{ii}) e_{jk} + e_{ii} g(e_{jk})$, for $j\neq k$.
\end{lemma}

\begin{proof}
We have two cases.

\textbf{Case1.} Let $i=j$. We prove that $f (e_{ik}e_{ii}) =g(e_{ik}) e_{ii} + e_{ik} h(e_{ii})=h(e_{ik}) e_{ii} + e_{ik} g(e_{ii})$, for $i\neq k$ which is equivalent to $f (e_{ij}e_{ii}) =g(e_{ij}) e_{ii} + e_{ij} h(e_{ii})=h(e_{ij}) e_{ii} + e_{ij} g(e_{ii})$, for $i\neq j$.

Equating the coefficients of $e_{1j}, e_{2j},\dots,e_{nj},e_{j1},e_{j2},\dots,e_{jn}$ from $0=f(e_{ij}\circ e_{ij})$, we have

\begin{equation}
\label{gh16y}
g_{1i}^{(ij)}+h_{1i}^{(ij)}=g_{2i}^{(ij)}+h_{2i}^{(ij)}=\dots=g_{ni}^{(ij)}+h_{ni}^{(ij)}=0,
\end{equation}

\begin{equation}
\label{gh17y}
g_{j1}^{(ij)}+h_{j1}^{(ij)}=g_{j2}^{(ij)}+h_{j2}^{(ij)}=\dots=g_{jn}^{(ij)}+h_{jn}^{(ij)}=0.
\end{equation}

Comparing the coefficients of $e_{1i},\dots,e_{i-1,i},e_{i+1,i},\dots,e_{ni}$ from \eqref{gh3y} and \eqref{gh4y},

\begin{equation}
\label{gh18y}
g_{1i}^{(ij)}=h_{1i}^{(ij)},\dots,g_{i-1,i}^{(ij)}=h_{i-1,i}^{(ij)},g_{i+1,i}^{(ij)}=h_{i+1,i}^{(ij)},\dots,g_{ni}^{(ij)}=h_{ni}^{(ij)}.
\end{equation}

Using \eqref{gh16y} and \eqref{gh18y},

\begin{equation}
\label{gh19y}
g_{1i}^{(ij)}=\dots=g_{i-1,i}^{(ij)}=g_{i+1,i}^{(ij)}=\dots=g_{ni}^{(ij)}=0,
\end{equation}

\begin{equation}
\label{gh20y}
h_{1i}^{(ij)}=\dots=h_{i-1,i}^{(ij)}=h_{i+1,i}^{(ij)}=\dots=h_{ni}^{(ij)}=0.
\end{equation}

Comparing coefficients of $e_{ii}$ from \eqref{gh4y} and \eqref{gh6y},
\begin{equation}
\label{gh21y}
h_{ji}^{(ii)}+2g_{ii}^{(ij)}=g_{ji}^{(jj)}.
\end{equation}

From \eqref{gh2y}, equating the coefficients of $e_{ji}$,
\begin{equation}
\label{gh22y}
h_{ji}^{(ii)}+g_{ji}^{(jj)}=0.
\end{equation}

By \eqref{gh21y} and \eqref{gh22y},
\begin{equation}
\label{gh23y}
g_{ii}^{(ij)}+h_{ji}^{(ii)}=0.
\end{equation}

From \eqref{gh2y}, equating the coefficients of $e_{j1},\dots,e_{j,i-1},e_{j,i+1},\dots,e_{jn}$,
\begin{equation}
\label{gh24y}
h_{j1}^{(ii)}=\dots=h_{j,i-1}^{(ii)}=h_{j,i+1}^{(ii)}=\dots=h_{jn}^{ii}=0.
\end{equation}

Using \eqref{gh19y},\eqref{gh23y} and \eqref{gh24y}, $g(e_{ij})e_{ii} + e_{ij} h(e_{ii})=0=f(e_{ij}e_{ii})$. Similarly, we can prove that $f(e_{ij}e_{ii})=h(e_{ij}) e_{ii} + e_{ij} g(e_{ii})$.

\textbf{Case 2.} Let $j\neq k$. Equating the coefficients of $e_{i1},e_{i2},\dots,e_{in}$ from $0=f(0)=f(e_{ii}\circ e_{jk})$,
\begin{equation}
\label{gh25y}
h_{i1}^{(jk)}=\dots=h_{i,k-1}^{(jk)}=h_{i,k+1}^{(jk)}=\dots=h_{in}^{jk}=0,
\end{equation}

\begin{equation}
\label{gh26y}
g_{ij}^{(ii)}+h_{ik}^{(jk)}=0.
\end{equation}

Using \eqref{gh13y},\eqref{gh25y} and \eqref{gh26y}, $g(e_{ii}) e_{jk} + e_{ii} h(e_{jk})=0=f (e_{ii}e_{jk})$. Similarly, $f (e_{ii}e_{jk})=h(e_{ii}) e_{jk} + e_{ii} g(e_{jk})$.
\end{proof}

\begin{lemma}
\label{lem8z}
$f (e_{ij}e_{kl}) =g(e_{ij}) e_{kl} + e_{ij} h(e_{kl})=h(e_{ij}) e_{kl} + e_{ij} g(e_{kl})$, for $i\neq j$ and $k\neq l$.
\end{lemma}

\begin{proof}
We have three cases.

\textbf{Case 1.} Let $j\neq k$. By Lemma \ref{lem7z}, $0=f(e_{ij}e_{kk})=g(e_{ij})e_{kl}+e_{ij}h(e_{kl})$. Now comparing the coefficients of $e_{1k},e_{2k},\dots,e_{nk}$,

\begin{equation}
\label{gh27y}
g_{1k}^{(ij)}=\dots=g_{i-1,k}^{(ij)}=g_{i+1,k}^{(ij)}=\dots=g_{nk}^{(ij)}=0,
\end{equation}

\begin{equation}
\label{gh28y}
g_{ik}^{(ij)}+h_{jk}^{(kk)}=0.
\end{equation}

Similarly, from $0=f(e_{jj}e_{kl})$, equating the coefficients of $e_{j1},e_{j2},\dots,e_{jn}$,

\begin{equation}
\label{gh29y}
h_{j1}^{(kl)}=\dots=h_{j,l-1}^{(kl)}=h_{j,l+1}^{(kl)}=\dots=h_{jn}^{(kl)}=0,
\end{equation}

\begin{equation}
\label{gh30y}
g_{jk}^{(jj)}+h_{jl}^{(kl)}=0.
\end{equation}

By Lemma \ref{lem6z}, $0=f(e_{jj}e_{kk})=g(e_{jj})e_{kk}+e_{jj}h(e_{kk})$. Now equating the coefficients of $e_{jk}$,
\begin{equation}
\label{gh31y}
g_{jk}^{(jj)}+h_{jk}^{(kk)}=0.
\end{equation}

By \eqref{gh28y},\eqref{gh30y} and \eqref{gh31y},
\begin{equation}
\label{gh32y}
g_{ik}^{(ij)}+h_{jl}^{(kl)}=0.
\end{equation}

So, by \eqref{gh27y},\eqref{gh29y} and \eqref{gh32y}, $g(e_{ij}) e_{kl} + e_{ij} h(e_{kl})=0=f (e_{ij}e_{kl})$. Similarly we have $f (e_{ij}e_{kl}) =h(e_{ij}) e_{kl} + e_{ij} g(e_{kl})$.

\textbf{Case 2.} Let $j=k$ and $i\neq l$. We prove $f(e_{kl}e_{ik})=g(e_{kl})e_{ik}+e_{kl}h(e_{ik})=h(e_{kl})e_{ik}+e_{kl}g(e_{ik})$. now from $0=f(e_{kl}e_{ii})=g(e_{kl})e_{ii}+e_{kl}h(e_{ii})$, equating the coefficients of $e_{1i},e_{2i},\dots,e_{ni}$,
\begin{equation}
\label{gh33y}
g_{1i}^{(kl)}=\dots=g_{k-1,i}^{(kl)}=g_{k+1,i}^{(kl)}=\dots=g_{ni}^{(kl)}=0,
\end{equation}

\begin{equation}
\label{gh34y}
g_{ki}^{(kl)}+g_{li}^{(ii)}=0.
\end{equation}

Equating the coefficients of $e_{l1},e_{l2},\dots,e_{ln}$ from $0=f(e_{ll}e_{ik})=g(e_{ll})e_{ik}+e_{ll}h(e_{ik})$,
\begin{equation}
\label{gh35y}
h_{l1}^{(ik)}=\dots=h_{l,k-1}^{(ik)}=h_{l,k+1}^{(ik)}=\dots=h_{ln}^{(ik)}=0,
\end{equation}

\begin{equation}
\label{gh36y}
g_{li}^{(ll)}+g_{lk}^{(ik)}=0.
\end{equation}

Equating the coefficients of $e_{li}$ from $0=f(e_{ll}e_{ii})=g(e_{ll})e_{ii}+e_{ll}h(e_{ii})$,
\begin{equation}
\label{gh37y}
g_{li}^{(ll)}+h_{li}^{(ii)}=0.
\end{equation}

By \eqref{gh34y},\eqref{gh36y} and \eqref{gh37y},
\begin{equation}
\label{gh38y}
g_{ki}^{(kl)}+h_{lk}^{(ik)}=0.
\end{equation}

So, by \eqref{gh33y},\eqref{gh35y} and \eqref{gh38y}, $g(e_{kl}) e_{ik} + e_{kl} h(e_{ik})=0=f (e_{kl}e_{ik})$. Similarly we have $f (e_{kl}e_{ik}) =h(e_{kl}) e_{ik} + e_{kl} g(e_{ik})$.

\textbf{Case 3.} Let $j=k$ and $i=l$. We prove $f(e_{ij}e_{ji})=g(e_{ij})e_{ji}+e_{ij}h(e_{ji})=h(e_{ij})e_{ji}+e_{ij}g(e_{ji})$. From $e_{ij}\circ e_{ji}=e_{ii}+e_{jj}$, using Lemma \ref{lem5z},
\begin{equation}
\label{gh39y}
g(e_{ij})\circ e_{ji}+e_{ij}\circ h(e_{ji})=g(e_{ii})e_{ii}+e_{ii}h(e_{ii})+g(e_{jj})e_{jj}+e_{jj}h(e_{jj}).
\end{equation}

Equating the coefficients of $e_{1i},\dots,e_{i-1,i},e_{ii},e_{i+1,i},\dots,e_{j-1,i},e_{j+1,i},\dots,e_{ni}$ and $e_{i1},\dots,e_{i,i-1},e_{i,i+1},\dots,e_{i,j-1},e_{i,j+1},\dots,e_{in}$ from both sides of \eqref{gh39y},

\begin{equation}
\label{gh40y}
\begin{aligned}
&g_{1j}^{(ij)}=g_{1i}^{(ii)},\dots,g_{i-1,j}^{(ij)}=g_{i-1,i}^{(ii)},g_{i+1,j}^{(ij)}=g_{i+1,i}^{(ii)},\\
&\dots,g_{j-1,j}^{(ij)}=g_{j-1,i}^{(ii)},g_{j+1,j}^{(ij)}=g_{j+1,i}^{(ii)},\dots,g_{nj}^{(ij)}=g_{ni}^{(ii)},
\end{aligned}
\end{equation}

\begin{equation}
\label{gh41y}
g_{ij}^{(ij)}+h_{ji}^{(ji)}=g_{ii}^{(ii)}+h_{ii}^{(ii)},
\end{equation}

\begin{equation}
\label{gh42y}
\begin{aligned}
&h_{j1}^{(ji)}=h_{i1}^{(ii)},\dots,h_{j,i-1}^{(ji)}=h_{i,i-1}^{(ii)},h_{j,i+1}^{(ji)}=h_{i,i+1}^{(ii)},\\
&\dots,h_{j,j-1}^{(ji)}=h_{i,j-1}^{(ii)},h_{j,j+1}^{(ji)}=h_{i,j+1}^{(ii)},\dots,h_{jn}^{(ji)}=h_{in}^{(ii)}.
\end{aligned}
\end{equation}

Equating the coefficients of $e_{jj}$ from $g(e_{ij})e_{jj}+e_{ij}h(e_{jj})=f(e_{ij})=g(e_{ii})e_{ij}+e_{ii}h(e_{ij})$,

\begin{equation}
\label{gh43y}
g_{jj}^{(ij)}=g_{ji}^{(ii)}.
\end{equation}

Equating the coefficients of $e_{jj}$ from $g(e_{jj})e_{ji}+e_{jj}h(e_{ji})=f(e_{ji})=g(e_{ji})e_{ii}+e_{ji}h(e_{ii})$,
\begin{equation}
\label{gh44y}
h_{jj}^{(ji)}=h_{ij}^{(ii)}.
\end{equation}

Using \eqref{gh40y}-\eqref{gh44y}, we get $g(e_{ij})e_{ji}+e_{ij}h(e_{ji})=g(e_{ii})e_{ii}+e_{ii}h(e_{ii})=f(e_{ii})=f(e_{ij}e_{ji})$. Similarly, we can prove that $f(e_{ij}e_{ji})=h(e_{ij})e_{ji}+e_{ij}g(e_{ji})$.
\end{proof}

\textbf{Proof of Theorem \eqref{thm2}:}
Let $f$ be a Jordan $\{g, h\}$-derivation on $\mathcal{M}_n(C)$. By Lemma \ref{lem5z}-\ref{lem8z}, we can prove that $f (xy) = g(x)y + xh(y) = h(x)y + xg(y) ~\text{for all}~ x, y \in \mathcal{M}_n(C)$.

\section*{Acknowledgement}
The authors are thankful to DST, Govt. of India for financial support and Indian Institute of Technology Patna for providing the research facilities.

\end{document}